\documentclass[11pt]{article}
 \usepackage[top=1.9in,bottom=1.9in,left=2in,right=2in]{geometry}
  \usepackage{amsmath,amssymb}
  \usepackage{latexsym}% defines $\Box$ for LaTeX2e
  \usepackage{tikz}
  \usetikzlibrary{decorations.pathreplacing}
  \usetikzlibrary{calc}
  \usepackage{setspace}
\usepackage{amsfonts} %\mathbb{R}, etc.
\usepackage{setspace}
\usepackage{subfigure}
\usepackage{url}
\usepackage{booktabs}
\usepackage{ifthen}
\usepackage{enumerate}

  \newcommand{\Z}{\mathbb{Z}}

  \newcommand{\hs}{\hspace*{\parindent}}

  \newcommand{\qed}{\hspace*{\fill} $\Box$\\}

  \newtheorem{theo}{\bfseries \hs Theorem}[section]

  \newtheorem{example}[theo]{\bfseries \hs Example}
  
  \numberwithin{equation}{section} % Automatically number equations within sections

\begin{document}
\title{On the Edge-balanced Index Sets\\ of Complete Bipartite Graphs}   
\author{Elliot Krop \thanks{Department of Mathematics, Clayton State University, (ElliotKrop@clayton.edu)}\and Keli Sikes \thanks{Department of Mathematics, Clayton State University, (ksikes1@student.clayton.edu)}} 
%\institute[CSU Mathematics]{Clayton State University \\ Department of Mathematics}
\date{\today} 

\maketitle

\begin {abstract}
Let $G$ be a graph with vertex set $V(G)$ and edge set $E(G)$, and $f$ be a $0-1$ labeling of $E(G)$ so that the absolute difference in the number of edges labeled $1$ and $0$ is no more than one. Call such a labeling $f$ \emph{edge-friendly}.   The \emph{edge-balanced index set} of the graph $G$, $EBI(G)$, is defined as the absolute difference between the number of vertices incident to more edges labeled $1$ and the number of vertices incident to more edges labeled $0$ over all edge-friendly labelings $f$. In 2009, Lee, Kong, and Wang \cite{LeeKongWang} found the $EBI(K_{l,n})$ for $l=1,2,3,4,5$ as well as $l=n$. We continue the investigation of the $EBI$ of complete bipartite graphs of other orders.
\\[\baselineskip] 2000 Mathematics Subject Classification: 05C78, 05C25
\\[\baselineskip]
\end {abstract}

 \section{Introduction}
 
 \subsection{Definitions}
For basic graph theoretic notation and definition see Diestel \cite{Diest}. The complete bipartite graph with $n$ and $m$ vertices partitioned into the first and second \emph{parts}, respectively, where no pair of vertices in the same part are adjacent and all other pairs are adjacent, is denoted $K_{n,m}$. The \emph{neighborhood} of a vertex $v$, denoted $N(v)$, is the set of all vertices adjacent to $v$. A \textit{labeling} of a graph $G$ with $H \subseteq G$, is a function $f : H \rightarrow \Z_2$, where $\Z_2 = \{ 0,1 \}$. If $H=E(G)$ $(H=V(G))$ and $f$ is surjective, call the labeling an \emph{edge labeling} (\emph{vertex labeling}). Denote by $f(e)$ the \emph{label} on edge $e$, and let $e_f(i) = \mbox{card}\{ e \in E : f(e)=i \}$. For a vertex $v$, let $N_0(v) = \{ u \in V: f(uv)=0 \}$ and $N_1(v) = \{ u \in V: f(uv)=1 \}$. A labeling $f$ is said to be \textit{edge-friendly} if $|e_f(0)-e_f(1)| \leq 1$. An edge-friendly labeling $f : E \rightarrow \Z_2$ induces a partial vertex labeling $f^+ : V \rightarrow \Z_2$ defined by $f^+(v) = 0$ if the number of $0$-edges incident to $v$ is more than the number of $1$-edges incident to $v$ and $f^+(v) = 1$ if the number of $1$-edges incident to $v$ is more than the number of $0$-edges. If the number of $1$-edges incident to $v$ is equal to the number of $0$-edges then $f^+(v)$ is not defined, and we say that $v$ is \emph{unlabeled}. For $i\in \Z_2$ let $v_f(i) = \mbox{card}\{v \in V : f^+(v) = i\}$. The \emph{edge-balanced index set} of the graph $G$, $EBI(G)$, is defined as $\{ |v_f(0) - v_f(1)| : \mbox{the edge labeling $f$ is edge-friendly} \}$.

\subsection{History and Motivation}

The study of balanced vertex labelings was introduced by S.M.~Lee, A.~Liu, and S.K.~Tan \cite{LeeLiuTan} in 1992. Three years later, M.~Kong and S.M.~Lee \cite{KongLeeEBI} continued work in balanced edge labelings. In the second labeling problem, the authors attempted to classify the $EBI(G)$, which proved to be quite difficult. For more information, J.~Gallian's dynamic survey of graph labeling problems in the \emph{Electronic Journal of Combinatorics} \cite{Gallian} provides an excellent overview of the subject. 

There have been many attempts to classify $EBI(G)$ by looking at different families of graphs. A common problem in attempting to find the $EBI(G)$ for a particular graph family is in the search for the maximal element in the set. Often after such an element is found, the authors produce an algorithm for the lesser terms. 

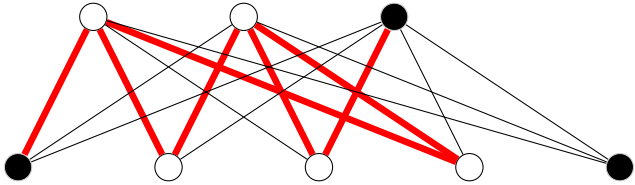
For complete bipartite graphs, Lee, Kong, and Wang \cite{LeeKongWang} found the $EBI(K_{l,n})$ for $l=1,2,3,4,5$ as well as $l=n$. In particular, they found $EBI(K_{n,n})=\{0,1,2,3,\dots, 2n-8\}$ for even $n$ and \{0,2,4,\dots, 2n-4\} for odd $n$. Little more is known for other graph families, as the edge-balance index set of regular graphs has been found for $2$-regular graphs, however, even for cubic graphs it is only known in some cases (M\"obius Ladders). We continue this inquiry by attempting to determine the $EBI$ of unknown cases of $K_{n,m}$, following the methods of \cite{KropLeeRaridan}.

\section{Complete bipartite graphs with two odd parts}

\subsection{An Example}

We consider some special cases of the form $K_{n,n-2}$ where $n$ is small.

 \begin{example}
 $EBI(K_{3,5})=\{0,2\}$
 \end{example}
 
 %\begin{figure}
 %\begin{center}
 \begin{figure}[ht]
\centering
\subfigure[$|v_f(0)-v_f(1)|=2$.]
{
 \begin{tikzpicture}
 
	 \node (v1) at (0,0) [circle,draw=black!20,fill=black!]{};
	 \node (v2) at (2,0) [circle,draw=black]{};
	 \node (v3) at (4,0) [circle,draw=black]{};
	 \node (v4) at (6,0) [circle,draw=black]{};
	 \node (v5) at (8,0) [circle,draw=black!20,fill=black!]{};
	 \node (u1) at (1,2) [circle,draw=black]{};
	 \node (u2) at (3,2) [circle,draw=black]{};
	 \node (u3) at (5,2) [circle,draw=black!20,fill=black!]{};

\draw [line width=2.5pt, color=red](u1) -- (v1);
\draw [line width=2.5pt, color=red](u1) -- (v2);
\draw [line width=2.5pt, color=red](u1) -- (v4);
\draw [line width=2.5pt, color=red](u2) -- (v2);
\draw [line width=2.5pt, color=red](u2) -- (v3);
\draw [line width=2.5pt, color=red](u2) -- (v4);
\draw [line width=2.5pt, color=red](u3) -- (v3);
\draw (u1) -- (v3);
\draw (u1) -- (v5);
\draw (u2) -- (v1);
\draw (u2) -- (v5);
\draw (u3) -- (v1);
\draw (u3) -- (v2);
\draw (u3) -- (v4);
\draw (u3) -- (v5);

\end{tikzpicture}
}
%\end{center}
 %\caption{$|v(0)-v(1)|=2$}
 %\end{figure}
 
 %\begin{figure}
 %\begin{center}
 \subfigure[$|v_f(0)-v_f(1)|=0$.]
{
 \begin{tikzpicture}
 
	 \node (v1) at (0,0) [circle,draw=black!20,fill=black!]{};
	 \node (v2) at (2,0) [circle,draw=black]{};
	 \node (v3) at (4,0) [circle,draw=black]{};
	 \node (v4) at (6,0) [circle,draw=black!20,fill=black!]{};
	 \node (v5) at (8,0) [circle,draw=black!20,fill=black!]{};
	 \node (u1) at (1,2) [circle,draw=black]{};
	 \node (u2) at (3,2) [circle,draw=black]{};
	 \node (u3) at (5,2) [circle,draw=black!20,fill=black!]{};

\draw [line width=2.5pt, color=red](u1) -- (v1);
\draw [line width=2.5pt, color=red](u1) -- (v2);
\draw [line width=2.5pt, color=red](u1) -- (v3);
\draw [line width=2.5pt, color=red](u2) -- (v2);
\draw [line width=2.5pt, color=red](u2) -- (v3);
\draw [line width=2.5pt, color=red](u2) -- (v4);
\draw [line width=2.5pt, color=red](u3) -- (v5);
\draw (u1) -- (v4);
\draw (u1) -- (v5);
\draw (u2) -- (v1);
\draw (u2) -- (v5);
\draw (u3) -- (v1);
\draw (u3) -- (v2);
\draw (u3) -- (v4);
\draw (u3) -- (v3);

\end{tikzpicture}
}
%\end{center}
 %\caption{$|v(0)-v(1)|=0$}
 \end{figure}
 
\subsection{Parts differing by two}

We begin by showing our general method in the simplest case.

\begin{theo}\label{simplecase}
$EBI(K_{n,n-2})=\{0,2, \dots, 2n-10, 2n-8\}$ for odd $n>5$
\end{theo}

\begin{proof}

We label edges of $G=K_{n,n-2}$ for maximum $EBI$. Vertices will be of two types, those which are incident to many more $0$-edges than $1$-edges, called \emph{dense}, and those which are incident to marginally more $1$-edges than $0$-edges, called \emph{sparse}. To find the maximum element of EBI, our goal is to minimize the number of dense vertices and maximize the number of sparse vertices. Call the vertices of the first part $A$ of $G$, $\{u, u_1, \dots, u_{n-3}\}$ and those of the second part $B$, $\{v,v', v_1, \dots, v_{n-2}\}$. Set $D=\{u,v,v'\}$ and $S=V(G)-D$.

The vertices $u,v,v'$ will be dense, so that all incident edges will be labeled zero, the rest of the vertices will be sparse. We will make half the sparse vertices in the larger part less sparse by one edge labeled $1$. 

Define a cyclic order on the vertices in $S$ so that for positive $i < n-3$ and $u_i$, the \emph{succeeding} vertex, $s(u_i)$, is $u_{i+1}$ and for $i=n-3$ the \emph{succeeding} vertex is $u_1$. For positive $i < n-2$ and $v_i$, the \emph{succeeding} vertex, $s(v_i)$, is $v_{i+1}$ and for $i=n-2$ the \emph{succeeding} vertex is $v_1$. For positive $i \leq n-3$ and $v_i$, the \emph{next} vertex is $u_{i}$, the \emph{next} vertex of $v_{n-2}$ is $u_1$. For any positive integer $k$, the $next_k(v_i)=s(s(\dots s(u_i)))$ where the number of iterations of $s$ is $k$. For $i \in \{1, \dots, n-2\}$ define the set $U_i=\{next(v_i), next_1(v_i), \dots, next_{\frac{n-1}{2}-1}(v_i)\}$ so that the cardinality of $U_i$ is $\frac{n-1}{2}$. 
For $i \in \{1, \dots, n-2\}$ define the set $U'_i=\{next(v_i), next_1(v_i), \dots, next_{\frac{n-1}{2}}(v_i)\}$ so that the cardinality of $U'_i$ is $\frac{n+1}{2}$. 

For $1\leq i \leq \frac{n-3}{2}$ label the edges in $S$ between the vertices in $U_i$ and $v_i$ by $1$ and the remaining edges incident to $v_i$ by $0$. For $\frac{n-1}{2}\leq i \leq n-2$ label the edges in $S$ between the vertices in $U'_i$ and $v_i$ by $1$ and the remaining edges incident to $v_i$ by $0$.

Notice that under this labeling, 
\[e(0)=2(n-3)+\frac{n-3}{2}\frac{n-5}{2}+\frac{n-1}{2}\frac{n-3}{2}\]
\[e(1)=\frac{n-3}{2}\frac{n+1}{2}+\frac{n-1}{2}\frac{n-1}{2}\]
and $e(0)-e(1)=-1$. Labeling the edges $uv$ and $uv'$ by $0$ produces the edge friendly labeling and the maximal EBI term.

\vspace{.2 in}

To verify that larger values of EBI are impossible, we consider fewer dense vertices and count the maximum number of $1$-edges, to show that the graph cannot be edge-friendly. Suppose $G$ is labeled with $2$ dense vertices, say $u, v$. The maximum number of $0$-edges is attained by making the rest of the vertices sparse. Therefore, 
\[e(1)\geq \frac{n-1}{2}(n-1)\]
\[e(0)\leq n-3 + \frac{n-3}{2}(n-1)\]

The difference in the above quantities is greater than $1$ for $n\geq 5$, so any labeling with $2$ dense vertices is not edge-friendly.

\vspace{.2 in}

To attain the smaller values of EBI we relabel the graph and switch $0$ and $1$ labels of pairs of edges incident to the same vertex. 

The vertices $u,v,v', v''$ will be dense, so that all incident edges will be labeled zero, the rest of the vertices will be sparse.

Define a cyclic order on the vertices in $S$ so that for positive $i < n-3$ and $u_i$, the \emph{succeeding} vertex, $s(u_i)$, is $u_{i+1}$ and for $i=n-3$ the \emph{succeeding} vertex is $u_1$. For positive $i < n-3$ and $v_i$, the \emph{succeeding} vertex, $s(v_i)$, is $v_{i+1}$ and for $i=n-3$ the \emph{succeeding} vertex is $v_1$. For positive $i \leq n-3$ and $u_i$, the \emph{next} vertex is $v_{i}$. For any positive integer $k$, the $next_k(u_i)=s(s(\dots s(v_i)))$ where the number of iterations of $s$ is $k$. For $i \in \{1, \dots, n-3\}$ define the set $V_i=\{next(u_i), next_1(u_i), \dots, next_{\frac{n-1}{2}}(u_i)\}$ so that the cardinality of $V_i$ is $\frac{n+1}{2}$. 

For $1\leq i \leq n-3$ label the edges in $S$ between the vertices in $V_i$ and $u_i$ by $1$ and the remaining edges incident to $v_i$ by $0$.

Next we count the number of edges labeled $0$ and those labeled $1$ to verify that the labeling is edge-friendly.

\[e(0)=4(n-3)+\frac{n-7}{2}(n-3)=\frac{n+1}{2}(n-3)\]
\[e(1)=\frac{n+1}{2}(n-3)\]

To complete the labeling, let $(u,v)$ and $(u,v')$ be labeled $0$ and $(u,v'')$ be labeled $1$.

Next we switch edge-labels so that every pairwise switch decreases $|v(1)-v(0)|$ by $2$.

\begin{itemize}
	\item$((v,u_2),(u_2,v_2)), \dots, ((v,u_{\frac{n-1}{2}}),(u_{\frac{n-1}{2}},v_{\frac{n-1}{2}}))$\\ reducing the 1-degree on $v_2, \dots, v_{\frac{n-1}{2}}$ by one.
  \item$((v',u_1),(u_1,v_2)), ((v',u_2),(u_2,v_3)), \dots,$\\ $((v',u_{\frac{n-3}{2}}),(u_{\frac{n-3}{2}},v_{\frac{n-1}{2}}))$ switching the labels on $v_2, \dots, v_{\frac{n-1}{2}}$ from $1$ to $0$.
  \item$((u,v_{\frac{n+1}{2}}),(v_{\frac{n+1}{2}},u_{\frac{n+1}{2}})), \dots, ((u,v_{n-4}),(v_{n-4},u_{n-4}))$ switching the labels on $u_{\frac{n+1}{2}},\dots, u_{n-4}$
  
  %\item$((u,v_2),(v_2,u_1)), ((u,v_3),(v_3,u_2))$ switching the labels on $u_1$ and $u_2$ from $1$ to $0$.
  \end{itemize}
Hence, we have produced the required EBI.

\qed\end{proof}

\subsection{General Case}

We use the method of the previous section to find the EBI of $K_{n,n-2a}$ for $a>1$.

\begin{theo}
%\mbox{}
   $EBI(K_{n,n-2a})=\{0,2, \dots, 2n-2a-8, 2n-2a-6\}$ for $1\leq a \leq \frac{n-3}{4}$ and odd $n>5$
  %$EBI(K_{n,n-2a})=\{0,2, \dots, 2n-10, 2n-2a-8\}$ for $\frac{n-3}{4}<a\leq \frac{n-4}{3}$ and odd $n>5$
  \end{theo}

\begin{proof}
We follow the structure and notation of the proof of Theorem \ref{simplecase}. For $a$ chosen as in the statement of the theorem and $2\leq c \leq 2a+1$, call the vertices of the first part $A$ of $G=K_{n,n-2a}$, $\{u, u_1, \dots, u_{n-2a-1}\}$ and those of the second part $B$, $\{v^1,v^2,\dots,v^c, v_1, \dots, v_{n-c}\}$. Set $D=\{u,v^1,v^2, \dots, v^c\}$ and $S=V(G)-D$. As before, the vertices of $D$ will be dense, so that all incident edges will be labeled $0$. The vertices of $S$ we will call sparse.

We verify that a smaller dense set where $c=1$ would not produce an edge-friendly labeling and hence that the EBI is no larger than $2n-2a-6$. We ignore the edge $(u,v)$, performing our edge-label count on the graph $G\backslash(u,v)$, and reintroduce it after the calculation. If $D=\{u,v\}$ with $u\in V(A)$ and $v\in V(B)$, then 
\[e(0)\leq n-2a-1 + \frac{n-2a-1}{2}(n-1)\]
\[e(1)\geq \frac{n-2a+1}{2}(n-1)\] and
\[e(1)-e(0) \geq 2a.\]
The edge $(u,v)$ can be labeled $0$, but nonetheless, the labeling is not edge-friendly for any $a>1$.

Define a cyclic order on the vertices in $S$ so that for positive $i < n-2a-1$ and $u_i$, the \emph{succeeding} vertex, $s(u_i)$, is $u_{i+1}$ and for $i=n-2a-1$ the \emph{succeeding} vertex is $u_1$. For positive $i < n-c$ and $v_i$, the \emph{succeeding} vertex, $s(v_i)$, is $v_{i+1}$ and for $i=n-c$ the \emph{succeeding} vertex is $v_1$. For positive $i \leq n-2a-1$ and $v_i$, the \emph{next} vertex is $u_{i}$. For $i=n-2a-1+j$, $1 \leq j \leq 2a+1-c$, the \emph{next} vertex of $v_i$ is $u_j$. For any positive integer $k$, the $next_k(v_i)=s(s(\dots s(u_i)))$ where the number of iterations of $s$ is $k$. 

We consider the labeling where sparse vertices in $S\cap B$ are incident to exactly $\frac{n-2a+1}{2}$ $1$-edges. We define $K$ to be the maximum number of times we can relabel an $0$-edge incident to every vertex in $S\cap B$ (but not to $u$) as a $1$-edge, so that $e(1)\leq e(0)+1$. Note that if we ignore all edges between $D\cap A$ and $D\cap B$,

\[e(1)=(n-c)(\frac{n-2a+1}{2}+K)\]
\[e(0)=c(n-2a-1)+(n-c)(\frac{n-2a-1}{2}-K)\]

Notice that $K$ cannot exceed the number of available $1$-edges at each vertex in $A\cap S$, so $K\leq \frac{n-2a-3}{2}$. By applying this upper bound to the above expressions to calculate $e(1)-e(0)$, we see that there are enough $1$-edges so long as $a\leq \frac{n-3}{2}$.

If $c$ is odd, then $K$ is the maximum integer so that $N=2ac-n(c-1) +2K(n-c)\leq 1$, where $N$ is the difference $e(1)-e(0)$. Note that switching a label on an edge from $0$ to $1$ increases $N$ by $2$.

For $i \in \{1, \dots, n-c\}$ define the set \[U_i=\{next(v_i), next_1(v_i), \dots, next_{\frac{n-2a-1}{2}+K}(v_i)\}.\] For $i \in \{1, \dots, n-c\}$ define the set \[U'_i=\{next(v_i), next_1(v_i), \dots, next_{\frac{n-2a+1}{2}+K}(v_i)\}.\]

\emph{Step 1}: If $\left\lfloor \frac{-N}{2}\right\rfloor=0$, then continue to \emph{Step 3}.

\emph{Step 2}: For $1\leq i \leq \left\lfloor \frac{-N}{2}\right\rfloor$ label the edges in $S$ between vertices in $U'_i$ and $v_i$ by $1$ and the remaining edges incident to $v_i$ by $0$.

\emph{Step 3}: For $\left\lfloor \frac{-N}{2}\right\rfloor+1 \leq i \leq n-c$ label the edges in $S$ between vertices in $U_i$ and $v_i$ by $1$ and the remaining edges incident to $v_i$ by $0$.

Concluding the labeling, we label the edges $(u,v^1), \dots, (u,v^{\frac{c-1}{2}})$ by $0$ and the edges $(u,v^{\frac{c+1}{2}}),\dots, (u,v^c)$ by $1$.

Notice that this labeling is edge-friendly by construction.

If $c$ is even, then we define $K,N,U_i, \mbox{ and } U'_i$ as above and repeat Steps 1-3. We conclude by labeling the edges of $(u,v^1), \dots, (u,v^{\frac{c}{2}})$ by $0$ and the edges $(u,v^{\frac{c+2}{2}}),\dots, (u,v^c)$ by $1$.

Again, this labeling is edge-friendly by construction.

For both the odd and even case, we justify that such a labeling is possible by counting the number of $1$-edges incident to any vertex $v_i$ in our labeling and showing that this does not exceed the number of vertices $u_j$ in the other part of $G$. Notice that it is enough to perform such a count for the case when $c=2a+1$, since the $1$-degree of $v_i$ is maximized. We must show
\[n-2a-1\geq \frac{n-2a+1}{2}+a(2a+1-n)\] which is true when
\[a\leq \frac{n}{2} -\frac{3}{4}\] and this inequality is satisfied by the range of $a$.

Thus, we have produced the following values of the EBI: $\{2n-6a-4,\dots, 2n-2a-6\}$. 

To show the lesser values, we switch $0$ and $1$ labels of pairs of edges incident to the same vertex in the case where $c=2a+1$ so that every pairwise switch decreases $|v(1)-v(0)|$ by $2$. By this procedure, we switch the labels on $n-3a-2$ vertices in $B\cap S$ from $1$ to $0$.

For ease of notation we define a cyclic order on the vertices in $B\cap D$ so that for positive $i < 2a+1$ and $v^i$, the \emph{succeeding} vertex, $s_1(v^i)$, is $v^{i+1}$ and for $i=2a+1$ the \emph{succeeding} vertex, $s_1(v^i)$, is $v^1$. For any positive integer $k$, we define the $k^{th}$ succceding vertex as $s_k(v^i)=s_1(s_{k-1}(v^i))$.

Note that
\begin{enumerate}
	\item $\deg_1(v_i)=\frac{n-2a+3}{2}+K \mbox{ for } 0\leq i \leq \left\lfloor \frac{-N}{2}\right\rfloor$\label{a}
  \item $\deg_1(v_i)=\frac{n-2a+1}{2}+K \mbox{ for } \left\lfloor \frac{-N}{2}\right\rfloor +1 \leq i \leq n-2a-1$\label{b}
\end{enumerate}

For \ref{a}. each vertex in $B\cap S$ requires $K+2$ switches of labels on incident edges from $0$ to $1$ so that the label on the vertex becomes $0$.
For \ref{b}. each vertex in $B\cap S$ requires $K+1$ switches of labels on incident edges from $0$ to $1$ so that the label on the vertex becomes $0$.

If $n-3a-2\leq \lfloor \frac{-N}{2}\rfloor$, we switch the labels on the following edges from $0$ to $1$:

\begin{itemize}
	\item Edges from vertices $(v^1, s_1(v^1), \dots, s_{K}(v^1))$ to $v_1$ reducing the 1-degree on $v_1$ by one.
	\item Edges from vertices $(s_{K+1}(v^1), \dots, s_{2K+1}(v^1))$ to $v_2$ reducing the 1-degree on $v_1$ by one.  
 \end{itemize}

 and so on until we reach vertex $v_{n-3a-2}$.

\vspace{.1 in}
 
If $n-3a-2> \lfloor \frac{-N}{2}\rfloor$, we switch the labels on the following edges from $0$ to $1$:

\begin{itemize}
	\item Edges from vertices $(v^1, s_1(v^1), \dots, s_{K}(v^1))$ to $v_1$ reducing the 1-degree on $v_1$ by one.
	\item Edges from vertices $(s_{K+1}(v^1), \dots, s_{2K+1}(v^1))$ to $v_2$ reducing the 1-degree on $v_2$ by one.  
 \end{itemize}
\vdots
\begin{itemize}
\item Edges from vertices $(s_{(\lfloor\frac{-N}{2}\rfloor-1)(K+1)}(v^1), \dots, s_{(\lfloor\frac{-N}{2}\rfloor)(K+1)-1}(v^1))$ to $v_{\lfloor\frac{-N}{2}\rfloor}$ reducing the 1-degree on $v_{\lfloor\frac{-N}{2}\rfloor}$ by one. 
\item Edges from vertices $(s_{(\lfloor\frac{-N}{2}\rfloor)(K+1)}(v^1), \dots, s_{(\lfloor\frac{-N}{2}\rfloor+1)(K+1)}(v^1))$ to $v_{\lfloor\frac{-N}{2}\rfloor+1}$ reducing the 1-degree on $v_{\lfloor\frac{-N}{2}\rfloor+1}$ by one. 
\end{itemize}
 
 and continuing with sets of succeeding $K+2$ vertices of $B\cap D$ until we relabel $v_1, \dots, v_{n-3a-2}$.
 
 Lastly, to show that such switches are possible, we calculate the number of edges incident with $B \cap D$ that we would need to switch, and show that this amount does not exceed the number of available edges. This is done by showing that
 $(K+2)(n-3a-2)\leq \frac{n-2a-1}{2}(2a+1)$ holds for $a\leq \frac{n-3}{4}$.
 
Hence, we have produced the required EBI.

\qed\end{proof}


\begin{thebibliography}{MMM}

\bibliographystyle{plain}
 
 \bibitem{Diest} R.~Diestel, \emph{Graph Theory, Third Edition}, Springer-Verlag, Heidelberg
 Graduate Texts in Mathematics, Volume 173, New York, 2005

 \bibitem{Gallian} J.A. Gallian, A dynamic survey of graph labeling, \emph{Electronic Journal of Combinatorics}, 17 (2010), \#DS6

 \bibitem{KongLeeEBI} M.~Kong and S.-M.~Lee, On edge-balanced graphs, \emph{Graph Theory,
Combinatorics and Algorithms}, 1 (1995), 711-722
 
 \bibitem{KropLeeRaridan} E.~Krop, S.-M.~Lee, C.~Raridan, On the edge-balanced index sets of product graphs, to appear in the Journal of the Indonesian Mathematical Society
 
 \bibitem{LeeKongWang} S.-M.~Lee, M.~Kong, and Y.-C.~Wang, On Edge-balance Index Sets of Some Complete k-partite Graphs, \emph{Congressus Numerantium} 196, (2009) 71-94)

 \bibitem{LeeLiuTan} S.-M.~Lee, A.~Liu, S.K.~Tan, On balanced graphs, \emph{Cong. Number.} 87 (1992), pp. 59-64

%\nocite{*}
\end{thebibliography}
\end{document}